\documentclass[12pt]{article}
\usepackage[utf8]{inputenc}
\usepackage{amsmath}
\usepackage{amssymb}
\usepackage{amsthm}
\usepackage{cite}
\usepackage{hyperref}
\usepackage{enumerate}
\usepackage[left=0.7in,right=0.7in,top=0.7in,bottom=0.7in]{geometry}
\usepackage{titlesec}
\titleformat*{\section}{\large\bfseries}
\newtheorem{theorem}{Theorem}[section]
\newtheorem{lemma}[theorem]{Lemma}
\newtheorem{corollary}[theorem]{Corollary}

\newtheorem{proposition}[theorem]{Proposition}

\title{A new characterization of generalized Browder's theorem and a Cline's formula for generalized Drazin-meromorphic inverses  }
\author{\large Anuradha Gupta$^1$ and Ankit Kumar$^2$\\ $^1$ {\small Department of Mathematics, Delhi College of Arts and Commerce,}\\ {\small University of Delhi, Netaji Nagar, New Delhi-110023, India.}\\{\small E-mail:dishna2@yahoo.in}\\{\small $^2$Department of Mathematics, University of Delhi, New Delhi-110007, India.}\\{\small E-mail 1995ankit13@gmail.com}}
\date{}
\begin{document}
\maketitle
\section*{Abstract}
In this paper, we give a new characterization of generalized Browder's theorem by considering equality between the generalized Drazin-meromorphic Weyl spectrum and the generalized Drazin-meromorphic spectrum.  Also, we generalize Cline's formula to the case of generalized Drazin-meromorphic invertibility under the assumption that $A^kB^kA^k=A^{k+1}$ for some positive integer $k$.\\
\textbf{Mathematics Subject Classification:} 47A10, 47A53.\\
\textbf{Keywords:} SVEP, generalized  Drazin-meromorphic  invertible, meromorphic operators, operator equation.
 \section{Introduction and Preliminaries}
 Throughout this paper, let $\mathbb{N}$ and $\mathbb{C}$ denote the set of natural numbers and complex numbers, respectively.  Let $B(X)$ denote the Banach algebra of all bounded linear operators acting on a complex Banach space $X$. For $T \in B(X)$,  we denote the spectrum of $T$, null space of $T$, range of $T$ and adjoint of $T$ by $\sigma(T)$, ker$(T)$, $R(T)$ and $T^*$, respectively.  For a subset $A$ of $\mathbb{C}$  the set of accumulation points of $A$ is denoted by acc$(A)$.  Let $\alpha(T)=$ dim ker($T$) and $\beta(T)= \mbox{codim}\thinspace R(T)$ be the nullity of $T$ and deficiency of $T$, respectively.  An operator $T \in B(X)$ is called a lower semi-Fredholm operator if $\beta(T) < \infty $. An operator $T \in B(X)$ is called an upper semi-Fredholm operator if $\alpha(T) < \infty $ and $R(T)$ is closed .  The  class of all lower semi-Fredholm operators (upper semi-Fredholm operators, respectively) is denoted by $\phi_{+} (X)$ ($\phi_{-}(X)$, respectively). An operator $T$ is called semi-Fredholm if it is upper or lower semi-Fredholm. For a semi-Fredholm operator $T \in B(X)$, the index of $T$ is defined by ind $(T$):= $\alpha(T)-\beta(T)$. The class of all Fredholm operators is defined by $\phi(X):=\phi_{+}(X)\cap \phi_{-}(X)$.  The class of all lower semi-Weyl operators (upper semi-Weyl operators, respectively) is defined by $W_{-} (X)=\{T\in\phi_{-} (X):$ ind $(T) \geq 0\}$  ($W_{+} (X)=\{T\in\phi_{+} (X):$ ind $(T) \leq 0\} $, respectively).  An operator $T \in B(X)$ is called Weyl if $T \in \phi(X)$ and ind $(T)=0$. The \emph{lower semi-Fredholm}, \emph{lower semi-Fredholm}, \emph{Fredholm}, \emph{lower semi-Weyl }, \emph{upper semi-Weyl} and  \emph{Weyl spectra}  are defined by 
 \begin{align*}
 \sigma_{lf}(T)&:=\{\lambda \in \mathbb{C}:\lambda I-T \thinspace \mbox{is not lower semi-Fredholm}\},\\
 \sigma_{uf}(T)&:=\{\lambda \in \mathbb{C}:\lambda I-T \thinspace  \mbox{is not upper semi-Fredholm}\},\\
 \sigma_{f}(T)&:=\{\lambda \in \mathbb{C}:\lambda I-T \thinspace  \mbox{is not Fredholm}\},\\
 \sigma_{lw}(T)&:=\{\lambda \in \mathbb{C}:\lambda I-T \thinspace \mbox{is not lower semi-Weyl}\},\\
 \sigma_{uw}(T)&:=\{\lambda \in \mathbb{C}:\lambda I-T \thinspace  \mbox{is not upper semi-Weyl}\},\\
 \sigma_{w}(T)&:=\{\lambda \in \mathbb{C}:\lambda I-T \thinspace  \mbox{is not Weyl}\}, \thinspace  \mbox{respectively}.
 \end{align*}
  A bounded linear operator $T$ is said to be bounded below if it is injective and $R(T)$ is closed. For $T \in B(X)$ the ascent denoted by  $p(T)$ is the smallest non negative integer $p$ such that ker$T^{p}= \mbox{ker}T^{p+1}$. If no such integer exists we set $p(T)= \infty$.  For $T \in B(X)$ the descent  denoted by $q(T)$ is the smallest non negative integer $q$ such that $R(T^q)=R(T^{q+1})$. If no such integer exists we set $q(T)= \infty$. By \cite[Theorem 1.20]{1} if both $p(T)$ and $q(T)$ are finite then $p(T)=q(T)$. An operator $T \in B(X)$ is called left Drazin invertible  if $p(T) < \infty$ and $R(T^{p+1})$ is closed. An operator $T \in B(X)$ is called right Drazin invertible  if $q(T) < \infty$ and $R(T^{q})$ is closed. Moreover, $T$ is called Drazin invertible if $p(T)=q(T) < \infty$.  An operator $T \in B(X)$ is called upper semi-Browder  if it is an upper semi-Fredholm  and $p(T) < \infty$. An operator $T \in B(X)$ is called lower semi-Browder  if it is an lower semi-Fredholm and $q(T) < \infty$. We say that an operator $T \in B(X)$ is  Browder if it is upper semi-Browder and lower semi-Browder. The \emph{lower semi-Browder}, \emph{upper semi-Browder} and \emph{Browder spectra} are defined by
  \begin{align*}
  \sigma_{lb}(T):& =\{\lambda \in \mathbb{C}: \lambda I -T \thinspace \mbox{is not lower semi-Browder} \},\\
  \sigma_{ub}(T):& =\{\lambda \in \mathbb{C}: \lambda I -T \thinspace \mbox{is not upper semi-Browder} \},\\
  \sigma_b(T):& =\{\lambda \in \mathbb{C}: \lambda I -T \thinspace \mbox{is not Browder} \}, \thinspace \mbox{respectively}.
  \end{align*} Clearly, every Browder operator is Drazin invertible.
 
An operator $T \in B(X)$ is said to possess the single-valued extension property (SVEP) at $\lambda_{0} \in \mathbb{C}$ if for every neighbourhood $V$ of $\lambda_{0}$ the only analytic function $f:V \rightarrow X$ which satisfies the equation $(\lambda I-T)f(\lambda)=0$ is the function $f=0$. If an operator $T$ has SVEP at every $\lambda \in \mathbb{C}$, then $T$ is said to have SVEP. Morever, the set of all points $\lambda \in \mathbb{C}$ such that $T$ does not have SVEP at $\lambda$ is an open set contained in interior of $\sigma(T)$. Therefore, if $T$ has SVEP at each point  of an open punctured disc $\mathbb{D} \setminus \{\lambda_{0}\}$ centered at $\lambda_{0}$, $T$ also has SVEP at $\lambda_{0}$.    
$$p(\lambda I-T) < \infty \Rightarrow  T \thinspace \mbox{has SVEP at} \thinspace \lambda$$
and 
$$q(\lambda I-T) < \infty \Rightarrow  T^{*} \thinspace \mbox{has SVEP at} \thinspace  \lambda.$$
An operator $T \in B(X)$ is called Riesz if $\lambda I-T$ is Browder for all $\lambda \in \mathbb{C} \setminus \{0\}$. An operator $T \in B(X)$ is called meromorphic if $\lambda I-T$ is Drazin invertible for all $\lambda \in \mathbb{C} \setminus \{0\}$. Clearly, every Riesz operator is meromorphic. A subspace $M$ of $X$ is said to be $T$-$invariant$ if $T(M) \subset M$. For a $T$-invariant subspace $M$ of $X$ we define $T_{M}:M \rightarrow M$ by $T_M (x)= T(x),x\in M$. We say $T$ is completely reduced by the pair $(M,N)$ (denoted by $(M,N) \in Red(T)$) if $M$ and $N$ are two closed $T$-invariant subspaces of $X$ such that $X=M \oplus N$.

An operator $T \in B(X)$ is called semi-regular if $R(T)$ is closed and ker$(T) \subset R(T^n)$ for every $n \in \mathbb{N}$. An operator $T \in B(X)$ is called nilpotent if $T^n =0$ for some $n \in \mathbb{N}$ and called quasi-nilpotent if $\vert \vert T^n \vert \vert^{\frac{1}{n}} \rightarrow 0$, i.e $\lambda I-T$ is invertible for all $\lambda \in \mathbb{C} \setminus \{0\}$.

For $T \in B(X)$  and a non negative integer $n$, define $T_{[n]}$ to be the restriction of $T$ to $T^n(X)$. If for some non negative integer $n$ the range space $T^n(X)$ is closed and $T_{[n]}$ is Fredholm (a lower semi B-Fredholm, an upper semi B-Fredholm, a lower semi B-Browder, an upper semi B-Browder, B-Browder, respectively) then $T$ is said to be B-Fredholm (a lower semi B-Fredholm, an upper semi B-Fredholm, a lower semi B-Browder, an upper semi B-Browder, B-Browder, respectively). For a semi B-Fredholm operator $T$, (see \cite{23}), the index of $T$ is defined as index of $T_{[n]}$. The \emph{lower semi B-Fredholm}, \emph{upper semi B-Fredholm} and \emph{B-Fredholm}, \emph{lower semi B-Browder}, \emph{upper semi B-Browder} and \emph{B-Browder spectra}  are defined by  
\begin{align*}
\sigma_{lsbf}(T)&:=\{\lambda \in \mathbb{C}: \lambda I-T \thinspace \mbox{is not lower semi B-Fredholm}\},\\
\sigma_{usbf}(T)&:=\{\lambda \in \mathbb{C}: \lambda I-T \thinspace \mbox{is not upper semi B-Fredholm}\},\\
\sigma_{bf}(T)&:=\{\lambda \in \mathbb{C}: \lambda I-T \thinspace \mbox{is not B-Fredholm}\}, \\
\sigma_{lsbb}(T)&:=\{\lambda \in \mathbb{C}: \lambda I-T \thinspace \mbox{is not lower semi B-Browder}\},\\
\sigma_{usbb}(T)&:=\{\lambda \in \mathbb{C}: \lambda I-T \thinspace \mbox{is not upper semi B-Browder}\},\\
\sigma_{bb}(T)&:=\{\lambda \in \mathbb{C}: \lambda I-T \thinspace \mbox{is not B-Browder}\}, \thinspace \mbox{respectively.}
\end{align*}
By \cite[Theorem 3.47]{1} an operator $T \in B(X)$ is upper semi B-Browder (lower semi B-Browder, B-Browder, respectively) if and only if $T$ is left Drazin invertible (right Drazin invertible, Drazin invertible, respectively).

An operator $T \in B(X)$ is called   a lower semi B-Weyl (an upper semi B-Weyl, respectively) if it is an lower semi B-Fredholm (an upper semi B-Fredholm, respectively)  having ind $(T)\leq 0$ (ind $(T) \geq 0$, respectively). An operator $T \in B(X)$ is called B-Weyl if it is B-Fredholm  and ind $(T)=0$. The \emph{lower semi B-Weyl}, \emph{upper semi B-Weyl} and  \emph{B-Weyl spectra} are defined by  
\begin{align*}
\sigma_{lsbw}(T)&:=\{\lambda \in \mathbb{C}: \lambda I-T \thinspace \mbox{is not lower semi B-Weyl}\},\\
\sigma_{usbw}(T)&:=\{\lambda \in \mathbb{C}: \lambda I-T \thinspace \mbox{is not upper semi B-Weyl}\},\\
\sigma_{bw}(T)&:=\{\lambda \in \mathbb{C}: \lambda I-T \thinspace \mbox{is not B-Weyl}\}, \thinspace \mbox{respectively.}
\end{align*}
It is known that (see \cite[Theorem 2.7]{23}) $T \in B(X)$ is B-Fredholm (B-Weyl, respectively) if there exists $(M,N) \in Red(T)$ such that $T_M$ is Fredholm (Weyl, respectively) and $T_N$ is nilpotent. Recently, (see \cite{29,31}) have generalized the class of B-Fredholm and B-Weyl operators and introduced the concept of pseudo B-Fredholm and pseudo B-Weyl operators. An operator $T \in B(X)$ is said to be pseudo B-Fredholm (pseudo B-Weyl, respectively) if there exists $(M,N) \in Red(T)$ such that $T_M$ is Fredholm (Weyl, respectively) and $T_N$ is quasi-nilpotent. The \emph{pseudo B-Fredholm} and \emph{pseudo B-Weyl spectra} are defined by
\begin{align*}
\sigma_{pBf}(T)&:=\{\lambda \in \mathbb{C}: \lambda I-T \thinspace \mbox{is not pseudo B-Fredholm}\},\\
\sigma_{pBw}(T)&:=\{\lambda \in \mathbb{C}: \lambda I-T \thinspace \mbox{is not pseudo B-Weyl}\}, \thinspace \mbox{respectively.}
\end{align*}
An operator $T$ is said to admit a \textit{generalized kato decomposition} ($GKD$), if there exists a pair $(M,N) \in Red(T)$ such that $T_M$ is semi-regular and $T_N$ is quasi-nilpotent. In the above definition if we assume $T_N$ to be nilpotent, then $T$ is said to be of Kato Type. (See \cite{10}) An operator is said to admit a \textit{Kato-Riesz decomposition} ($GKRD$), if there exists a pair $(M,N) \in Red(T)$ such that $T_M$ is semi-regular and $T_N$ is Riesz. 

Recently, \v{Z}ivkovi\'{c}-Zlatanovi\'{c} and  Duggal \cite{30} introduced the notion of generalized Kato-meromorphic decomposition. An operator $T \in B(X)$ is said to admit a \emph{generalized Kato-meromorphic decomposition} ($GKMD$), if there exists a pair $(M,N) \in Red(T)$ such that $T_M$ is semi-regular and $T_N$ is meromorphic. For $T \in B(X)$ the \textit{generalized Kato-meromorphic spectrum} is defined by $$\sigma_{gKM}(T):= \{\lambda \in \mathbb{C}:\lambda I-T \thinspace  \mbox{does not admit a GKMD}\}.$$ 
Recall that an operator $T \in B(X)$ is said to be Drazin invertible if there exists $S \in B(X)$ such that $TS=ST$, $STS=S$ and $TST-T$ is nilpotent. This definition is equivalent to the fact that there exist of  a pair $(M,N) \in Red(T)$ such that $T_M$ is invertible and $T_N$ is nilpotent.  Koliha \cite{8} generalized this concept by replacing the third condition with $TST-T$ is quasi-nilpotent. An operator is said to be generalized Drazin invertible if there exist  a pair $(M,N) \in Red(T)$ such that $T_M$ is invertible and $T_N$ is quasi-nilpotent. The \emph{generalized Drazin spectrum} is defined by
$$\sigma_{gD}(T):= \{ \lambda \in \mathbb{C}: \lambda I-T \thinspace  \mbox{is not generalized Drazin invertible}\}.$$
 Recently, \v{Z}ivkovi\'{c}-Zlatanovi\'{c}\ and Cvetkovi\'{c} \cite{10}  introduced the concept of  generalized Drazin-Riesz invertible by replacing the third condition with $TST-T$ is Riesz. They proved that the an operator $T \in B(X)$ is generalized Drazin-Riesz invertible if and only if there exists  a pair $(M,N) \in Red(T)$ such that $T_M$ is invertible and $T_N$ is Riesz. An operator $T \in B(X)$ is called generalized Drazin-Riesz bounded below (surjective, respectively) if there exists a pair $(M,N) \in Red(T)$ such that $T_M$ is bounded below (surjective, respectively) and $T_N$ is Riesz. The \emph{generalized Drazin-Riesz bounded below}, \emph{generalized Drazin-Riesz surjective} and \emph{generalized Drazin-Riesz invertible spectra} are defined by
 \begin{align*}
\sigma_{gDR\mathcal{J}}(T)&:=\{\lambda \in \mathbb{C}: \lambda I-T \thinspace \mbox{is not generalized Drazin-Riesz bounded below}\},\\
\sigma_{gDR\mathcal{Q}}(T)&:=\{\lambda \in \mathbb{C}: \lambda I-T \thinspace \mbox{is not generalized Drazin-Riesz surjective}\},\\
\sigma_{gDR}(T)&:=\{\lambda \in \mathbb{C}: \lambda I-T \thinspace \mbox{is not generalized Drazin-Riesz invertible}\}, \thinspace \mbox{respectively.}
\end{align*}
Also, they introduced the notion  of operators which are direct sum of a Riesz  and a Fredholm (lower (upper) semi-Fredholm, lower (upper) semi-Weyl, Weyl). An operator is called generalized Drazin-Riesz Fredholm (generalized Drazin-Riesz lower (upper) semi-Fredholm, generalized Drazin-Riesz lower (upper) semi-Weyl, generalized Drazin-Riesz Weyl, respectively) if there exists $(M,N) \in Red(T)$ such that $T_M$ is  Fredholm (lower (upper) semi-Fredholm, lower (upper) semi-Weyl, Weyl, respectively) and $T_N$ is Riesz. The \emph{The generalized Drazin-Riesz lower (upper) semi-Fredholm}, \emph{generalized Drazin-Riesz Fredholm}, \emph{generalized Drazin-Riesz  upper(lower) semi-Weyl} and \emph{generalized Drazin-Riesz Weyl spectra}, are defined by
\begin{align*}
\sigma_{gDR\phi_{-}}(T)&:=\{\lambda \in \mathbb{C}: \lambda I-T \thinspace \mbox{is not generalized Drazin-Riesz lower semi-Fredholm}\},\\
\sigma_{gDR\phi_{+}}(T)&:=\{\lambda \in \mathbb{C}: \lambda I-T \thinspace \mbox{is not generalized Drazin-Riesz upper semi-Fredholm}\},\\
\sigma_{gDR\phi}(T)&:=\{\lambda \in \mathbb{C}: \lambda I-T \thinspace \mbox{is not generalized Drazin-Riesz Fredholm}\},\\
\sigma_{gDRW_{-}}(T)&:=\{\lambda \in \mathbb{C}: \lambda I-T \thinspace \mbox{is not generalized Drazin-Riesz lower semi-Weyl}\},\\
\sigma_{gDRW_{+}}(T)&:=\{\lambda \in \mathbb{C}: \lambda I-T \thinspace \mbox{is not generalized Drazin-Riesz upper semi-Weyl}\},\\
\sigma_{gDRW}(T)&:=\{\lambda \in \mathbb{C}: \lambda I-T \thinspace \mbox{is not generalized Drazin-Riesz Weyl}\},
 \thinspace \mbox{respectively.}
\end{align*}
 Also, \v{Z}ivkovi\'{c}-Zlatanovi\'{c}   and Duggal   \cite{30} introduced the notion of  generalized Drazin-meromorphic invertible by replacing the third condition with $TST-T$ is meromorphic. They proved that the an operator $T \in B(X)$ is generalized Drazin-meromorphic invertible if and only if there exists  a pair $(M,N) \in Red(T)$ such that $T_M$ is invertible and $T_N$ is meromorphic. An operator $T \in B(X)$ is said to be  generalized Drazin-meromorphic  bounded below  (surjective, respectively) if there exists a pair $(M,N) \in Red(T)$ such that $T_M$ is bounded below  (surjective, respectively) and $T_N$ is meromorphic. The \emph{generalized Drazin-meromorphic  bounded below}, \emph{generalized Drazin-meromorphic  surjective} and \emph{generalized Drazin-meromorphic invertible spectra} are defined by
 \begin{align*}
\sigma_{gDM\mathcal{J}}(T)&:=\{\lambda \in \mathbb{C}: \lambda I-T \thinspace \mbox{is not generalized Drazin-meromorphic bounded below}\}\\
\sigma_{gDM\mathcal{Q}}(T)&:=\{\lambda \in \mathbb{C}: \lambda I-T \thinspace \mbox{is not generalized Drazin-meromorphic surjective}\},\\
\sigma_{gDM}(T)&:=\{\lambda \in \mathbb{C}: \lambda I-T \thinspace \mbox{is not generalized Drazin-meromorphic invertible}\}, \thinspace \mbox{respectively.}
\end{align*}
Also, they introduced the notion  of operators which are direct sum of a meromorphic  and Fredholm (lower (upper) semi-Fredholm, lower (upper) semi-Weyl, Weyl). An operator is called generalized Drazin-meromorphic Fredholm (generalized Drazin-meromorphic lower (upper) semi-Fredholm,  generalized Drazin-meromorphic lower (upper) semi-Weyl, generalized Drazin-meromorphic Weyl) if there exists $(M,N) \in Red(T)$ such that $T_M$ is  Fredholm (lower (upper) semi-Fredholm, lower (upper) semi-Weyl, Weyl) and $T_N$ is Riesz. The \emph{generalized Drazin-meromorphic lower (upper) semi-Fredholm}, \emph{generalized Drazin-meromorphic Fredholm}, \emph{generalized Drazin-meromorphic  lower (upper) semi-Weyl} and \emph{generalized Drazin-meromorphic Weyl spectra} are  defined by
\begin{align*}
\sigma_{gDM\phi_{-}}(T)&:=\{\lambda \in \mathbb{C}: \lambda I-T \thinspace \mbox{is not generalized Drazin-meromorphic lower semi-Fredholm}\},\\
\sigma_{gDM\phi_{+}}(T)&:=\{\lambda \in \mathbb{C}: \lambda I-T \thinspace \mbox{is not generalized Drazin-meromorphic upper semi-Fredholm}\},\\
\sigma_{gDM\phi}(T)&:=\{\lambda \in \mathbb{C}: \lambda I-T \thinspace \mbox{is not generalized Drazin-meromorphic Fredholm}\},\\
\sigma_{gDMW_{-}}(T)&:=\{\lambda \in \mathbb{C}: 
\lambda I-T \thinspace \mbox{is not generalized Drazin-meromorphic lower semi-Weyl}\},\\
\sigma_{gDMW_{+}}(T)&:=\{\lambda \in \mathbb{C}: \lambda I-T \thinspace \mbox{is not generalized Drazin-meromorphic upper semi-Weyl}\},\\
\sigma_{gDMW}(T)&:=\{\lambda \in \mathbb{C}: \lambda I-T \thinspace \mbox{is not generalized Drazin-meromorphic Weyl}\},
 \thinspace \mbox{respectively.}
\end{align*}
From \cite{10,30} we have
\begin{align*}
\sigma_{gD*\phi}(T)&=\sigma_{gD*\phi_{+}}(T) \cup \sigma_{gD*\phi_{-}}(T),\\
\sigma_{gK*}(T) & \subset \sigma_{gD*\phi_{+}}(T) \subset \sigma_{gD*W_{+}} (T) \subset \sigma_{gD*\mathcal{J}}(T),\\
\sigma_{gK*}(T) & \subset \sigma_{gD*\phi_{-}}(T) \subset \sigma_{gD*W_{-}}(T)  \subset \sigma_{gD*\mathcal{Q}}(T),\\
\sigma_{gK*}(T) & \subset \sigma_{gD*\phi}(T) \subset \sigma_{gD*W} \subset \sigma_{gD*}(T),
\end{align*}
where $*$ stands for Riesz or meromorphic operators.

Recall that an operator $T$ satisfies Browder's  theorem if $\sigma_b(T)=\sigma_w(T)$ and generalized Browder's theorem if $\sigma_{bb}(T)=\sigma_{bw}(T)$. Amouch et al. \cite{25} and Karmouni and Tajmouati \cite{27} gave a new characterization of Browder's theorem using spectra arised from Fredholm theory and Drazin invertibilty. Motivated by them, we give a new  characterization of operators satisfying generalized Browder's theorem. We prove that an operator $T$ satisfies generalized Browder's theorem if and only if $\sigma_{gDMW}(T)=\sigma_{gDM}(T)$. In the last section, we generalize the Cline's formula for the case of generalized Drazin-meromorphic invertibility under the assumption that $A^kB^kA^k=A^{k+1}$ for some positive integer $k$.

 \section{Main Results}
 The following result will be used in the sequel:
 \begin{theorem}\label{pretheorem1}
 \emph{\cite[Theorem 2.1]{30}} Let $T \in B(X)$, then $T$ is generalized Drazin-meromorphic upper semi-Weyl (lower semi-Weyl, upper semi-Fredholm, lower semi-Fredholm, Weyl, respectively) if and only if $T$ admits a $GKMD$ and $0 \notin \emph{acc}\sigma_{usbw}(T) \thinspace (\emph{acc}\sigma_{lsbw}(T),\emph{acc}\sigma_{usbf}(T), \emph{acc}\sigma_{lsbf}(T), \emph{acc}\sigma_{bw}(T)$, respectively).
 \end{theorem}
\begin{proposition}\label{theorem1}
Let $T \in B(X)$, then $\sigma_{gDM\mathcal{J}}(T)=\sigma_{gDMW_{+}}(T)$ if and only if $T$ has SVEP at every $\lambda \notin \sigma_{gDMW_{+}}(T)$.
\end{proposition}
\begin{proof}
Suppose that $\sigma_{gDM\mathcal{J}}(T)=\sigma_{gDMW_{+}}(T)$. Let $\lambda \notin\sigma_{gDMW_{+}}(T)$, then $\lambda I-T$ generalized Drazin-meromorphic bounded below. Therefore, by  \cite[Theorem 2.5]{30} $T$ has SVEP at $\lambda$. Conversely, suppose that $T$ has SVEP at every $\lambda \notin \sigma_{gDMW_{+}}(T)$. It suffices to show that $\sigma_{gDM\mathcal{J}}(T) \subset \sigma_{gDMW_{+}}(T)$. Let $\lambda \notin \sigma_{gDMW_{+}}(T)$ which implies that $\lambda I-T$ is generalized Drazin-meromorphic upper semi-Weyl. Therefore, by Theorem \ref{pretheorem1} $\lambda I-T$ admits a $GKMD$. Thus, there exists $(M,N) \in Red(\lambda I-T)$ such that $(\lambda I-T)_M$ is semi-regular and $(\lambda I-T)_N$ is meromorphic. Since $T$ has SVEP at every $\lambda \notin  \sigma_{gDMW_{+}}(T)$, $(\lambda I-T)$ has SVEP at $0$. As SVEP at a point is inherited by the restrictions on closed invariant subspaces,  $(\lambda I-T)_M$ has SVEP at $0$. Therefore, by   \cite[Theorem 2.91]{1} $(\lambda I-T)_M$ is bounded below. Thus, by \cite[Theorem 2.6]{30} we have $\lambda I-T$ is generalized Drazin-meromorphic bounded below. Hence, $\lambda \notin \sigma_{gDM\mathcal{J}}(T)$. 
\end{proof}
\begin{proposition}\label{theorem2}
Let $T \in B(X),$ then $\sigma_{gDM \mathcal{Q}}(T)=\sigma_{gDMW_{-}}(T)$ if and only if $T^{*}$ has SVEP at every $\lambda \notin \sigma_{gDMW_{-}}(T)$.
\end{proposition}
\begin{proof}
Suppose that $\sigma_{gDM\mathcal{Q}}(T)=\sigma_{gDMW_{-}}(T)$. Let $\lambda \notin \sigma_{gDMW_{-}}(T)$, then $\lambda I-T$ generalized Drazin-meromorphic surjective. Therefore, by \cite[Theorem 2.6]{30} $T^{*}$ has SVEP at $\lambda$. Conversely, suppose that $T^{*}$ has SVEP at every $\lambda \notin \sigma_{gDMW_{-}}(T)$. It suffices to show that $\sigma_{gDM\mathcal{Q}}(T) \subset \sigma_{gDMW_{-}}(T)$.  Let $\lambda \notin \sigma_{gDMW_{-}}(T)$ which implies that $\lambda I-T$ is generalized Drazin-meromorphic lower semi-Weyl. Then by Theorem \ref{pretheorem1} $\lambda I-T$ admits a $GKMD$ and $\lambda \notin \mbox{acc} \sigma_{lsbw}(T)$. Since $T^{*}$ has SVEP at every $\lambda \notin \sigma_{gDMW_{-}}(T)$ and  $ \sigma_{gDMW_{-}}(T) \subset \sigma_{lw}(T)$ then $T^{*}$ has SVEP at every $\lambda \notin \sigma_{lw}(T)=\sigma_{uw}(T^{*})$.  Therefore, by \cite[Theorem 5.27]{1} we have  $\sigma_{lw}(T)=\sigma_{uw}(T^{*})=\sigma_{ub}(T^{*})= \sigma_{lb}(T)$. Thus, by \cite[Theorem 5.38]{1} we have $\sigma_{lsbw}(T)= \sigma_{lsbb}(T)$. This implies that  $\lambda \notin \mbox{acc} \sigma_{lsbb}(T)$. Therefore, by  \cite[Theorem 2.6]{30} $\lambda I-T$ is generalized Drazin-meromorprhic surjective and it follows that $\lambda \notin \sigma_{gDM\mathcal{Q}}(T)$.  
\end{proof}
\begin{corollary}\label{corollary1}
Let $T \in B(X),$ then $\sigma_{gDM}(T)=\sigma_{gDMW}(T)$ if and only if $T$ and $T^{*}$ have SVEP at every $\lambda \notin \sigma_{gDMW}(T)$.
\end{corollary}
\begin{proof}
Suppose that $\sigma_{gDM}(T)=\sigma_{gDMW}(T).$ Let $\lambda \notin \sigma_{gDMW}(T)$, then $\lambda I-T$ is generalized Drazin-meromorphic invertible. Therefore, by  \cite[Theorem 2.4]{30}  $T$ and $T^{*}$ have SVEP at $\lambda$. Conversely, let $\lambda \notin \sigma_{gDMW}(T)=\sigma_{gDMW_{+}}(T) \cup \sigma_{gDMW_{-}}(T)$. Then by proofs of Theorem \ref{theorem1} and Theorem \ref{theorem2} we have $ \lambda \notin \sigma_{gDM\mathcal{J}}(T) \cup \sigma_{gDM\mathcal{Q}}(T)=\sigma_{gDM}(T).$
\end{proof}
\begin{theorem}\label{theorem3}
Let $T \in B(X),$ then following statements are equivalent:

(i) $\sigma_{gDM}(T)=\sigma_{gDMW}(T)$,

(ii) $T$ or $T^{*}$ have SVEP at every $\lambda \notin \sigma_{gDMW}(T)$.
\end{theorem}
\begin{proof}
Suppose that $T$ has SVEP at every $\lambda \notin \sigma_{gDRW}(T)$. It suffices to prove that $\sigma_{gDM}(T) \subset \sigma_{gDMW}(T)$. Let $\lambda \notin  \sigma_{gDMW}(T)$ then $\lambda I-T$  admits a $GKMD$ and $\lambda \notin \mbox{acc} \sigma_{bw}(T)$. Since $\sigma_{gDRW}(T) \subset \sigma_{bw}(T)$,  $T$ has SVEP at every $\lambda \notin \sigma_{bw}(T)$. Therefore, $\sigma_{bw}(T)=\sigma_{bb}(T)$. Thus, $\lambda \notin \mbox{acc} \sigma_{bb}(T)$ which implies that $\lambda I-T$ is generalized Drazin-meromorphic invertible.

Now suppose that $T^{*}$ has SVEP at every  $\lambda \notin \sigma_{gDRW}(T)$. Since $\sigma_{bb}(T)=\sigma_{bb}(T^*)$ and $\sigma_{bw}(T)=\sigma_{bw}(T^*)$ we have $\sigma_{gDR}(T)=\sigma_{gDRW}(T)$. The converse is an immediate consequence of Corollary \ref{corollary1}.
\end{proof}
Recall that an operator $T \in B(X)$ is said satisfy  generalized a-Browder's theorem if $\sigma_{usbb}(T)=\sigma_{usbw}(T)$. An operator $T \in B(X)$ satisfies a-Browder's theorem if $\sigma_{ub}(T)=\sigma_{uw}(T)$. By \cite[Theorem 2.2]{21} we know that a-Browder's theorem is equivalent to generalized a-Browder's theorem.
\begin{theorem}\label{theorem4}
Let $T \in B(X)$, then the following holds:

(i) generalized a-Browder's theorem holds for $T$ if and only if $\sigma_{gDM\mathcal{J}}(T)=\sigma_{gDMW_{+}}(T)$,

(ii) generalized a-Browder's theorem holds for $T^*$ if and only if $\sigma_{gDM\mathcal{Q}}(T)=\sigma_{gDMW_{-}}(T)$,

(iii)generalized  Browder's theorem holds for $T$ if and only if $\sigma_{gDM}(T)=\sigma_{gDMW}(T)$.
\end{theorem}
\begin{proof}
(i) Suppose that generalized a-Browder's theorem holds for $T$ which implies that  $\sigma_{usbb}(T)=\sigma_{usbw}(T)$. It suffices to prove that $\sigma_{gDM\mathcal{J}}(T) \subset \sigma_{gDMW_{+}}(T)$. Let $\lambda \notin \sigma_{gDMW_{+}}(T)$, then $\lambda I-T$ is generalized Drazin-meromorphic upper semi-Weyl. By Theorem \ref{pretheorem1} it follows that $\lambda I-T$ admits a $GKMD$ and $\lambda \notin \mbox{acc} \sigma_{usbw}(T)$. This gives $\lambda \notin \mbox{acc} \sigma_{usbb}(T)$. Therefore, by  \cite[Theorem 2.5]{30} $\lambda I-T$ is generalized Drazin-meromorphic bounded below which gives $\lambda \notin \sigma_{gDM\mathcal{J}}(T)$.  Conversely, suppose that $\sigma_{gDM\mathcal{J}}(T)=\sigma_{gDMW_{+}}(T)$. Using Proposition \ref{theorem1} we deduce that $T$ has SVEP at every $\lambda \notin \sigma_{gDMW_{+}}(T)$. Since $\sigma_{gDMW_{+}}(T) \subset \sigma_{uw}(T)$, $T$ has SVEP at every $\lambda \notin \sigma_{uw}(T)$. By \cite[Theorem 5.27]{1} $T$ satisfies a-Browder's theorem.  Therefore,  generalized a-Browder's theorem holds for $T$.\\
(ii) Suppose that generalized a-Browder's theorem holds for $T^{*}$ which implies that $\sigma_{lsbb}(T)=\sigma_{lsbw}(T)$. It suffices to prove that $\sigma_{gDM\mathcal{Q}}(T) \subset \sigma_{gDMW_{-}}(T)$. Let $\lambda \notin \sigma_{gDMW_{-}}(T)$, then $\lambda I-T$ is generalized Drazin-meromorphic lower semi-Weyl. By Theorem \ref{pretheorem1} it follows that $\lambda I-T$ admits a $GKMD$ and $\lambda \notin \mbox{acc} \sigma_{lsbw}(T)$. This gives $\lambda \notin \mbox{acc} \sigma_{lsbb}(T)$. Therefore, by  \cite[Theorem 2.6]{30} $\lambda I-T$ is generalized Drazin-meromorphic surjective which gives $\lambda \notin \sigma_{gDM\mathcal{Q}}(T)$.  Conversely, suppose that $\sigma_{gDM\mathcal{Q}}(T)=\sigma_{gDMW_{-}}(T)$. Using Proposition \ref{theorem2} we deduce that $T^*$ has SVEP at every $\lambda \notin \sigma_{gDMW_{-}}(T)$. Since $\sigma_{gDMW_{-}}(T) \subset \sigma_{lw}(T)$, $T^*$ has SVEP at every $\lambda \notin \sigma_{lw}(T)=\sigma_{uw}(T^{*})$. Therefore,  generalized a-Browder's theorem holds for $T^*$.\\ 
(iii) Suppose that generalized Browder's theorem holds for $T$ which implies that $\sigma_{bb}(T)=\sigma_{bw}(T)$. It suffices to prove that $\sigma_{gDM}(T) \subset \sigma_{gDMW}(T)$. Let $\lambda \notin \sigma_{gDMW}(T)$, then $\lambda I-T$ is generalized Drazin-meromorphic Weyl. By Theorem \ref{pretheorem1} it follows that $\lambda I-T$ admits a $GKMD$ and $\lambda \notin \mbox{acc} \sigma_{bw}(T)$. This gives $\lambda \notin \mbox{acc} \sigma_{bb}(T)$. Therefore, by  \cite[Theorem 2.4]{30} $\lambda I-T$ is generalized Drazin-meromorphic invertible which gives $\lambda \notin \sigma_{gDM}(T)$.  Conversely, suppose that $\sigma_{gDM}(T)=\sigma_{gDMW}(T)$.  Using Proposition \ref{corollary1} we deduce that  $T$ and $T^*$ have SVEP at every $\lambda \notin \sigma_{gDMW}(T)$. Since $\sigma_{gDMW}(T) \subset \sigma_{bw}(T)$, $T$ and $T^*$ have SVEP at every $\lambda \notin \sigma_{bw}(T)$. Therefore, by \cite[Theorem 5.14]{1} generalized Browder's theorem holds for $T$. 
\end{proof}
Using Theorem \ref{theorem4}, \cite[Theorem 2.3]{32}, \cite[Theorem 2.1]{21}, \cite[Proposition 2.2]{22} and \cite[Theorem 2.6]{27} we have the following theorem:
\begin{theorem}
Let $T \in B(X)$, then the following statements are equivalent:

(i)  Browder's theorem holds for $T$,

(ii) Browder's theorem holds for $T^{*}$,

(iii) $T$ has SVEP at every $\lambda \notin \sigma_{w}(T)$,

(iv) $T^{*}$ has SVEP at every $\lambda \notin \sigma_{w}(T)$.

(v) $T$ has SVEP at every $\lambda \notin \sigma_{bw}(T)$.

(vi) generalized Browder's theorem holds for $T$.

(vii) $T$ or $T^{*}$ has SVEP at every $\lambda \notin \sigma_{gDRW}(T).$

(viii) $\sigma_{gDR}(T)=\sigma_{gDRW}(T)$, 

(ix) $T$ or $T^{*}$ has SVEP at every $\lambda \notin \sigma_{gDMW}(T)$,

(x) $\sigma_{gDM}(T)=\sigma_{gDMW}(T)$,

(xi) $\sigma_{gD}(T)=\sigma_{pBW}(T)$.
\end{theorem}
Using \cite[Theorem 2.2]{21} and \cite[Theorem 2.7]{27} a similar result for a-Browder's theorem can be stated as follows:
\begin{theorem}
Let $T \in B(X)$, then the following statements are equivalent:

(i) a-Browder's theorem holds for $T$,

(ii) generalized a-Browder's theorem holds for $T$,

(iii) $T$ has SVEP at every $\lambda \notin \sigma_{gDRW_{+}}(T)$,

(iv) $\sigma_{gDR\mathcal{J}}(T)=\sigma_{gDRW_{+}}(T)$,

(v) $T$  has SVEP at every $\lambda \notin \sigma_{gDMW_{+}}(T)$,

(vi) $\sigma_{gDM\mathcal{J}}(T)=\sigma_{gDMW_{+}}(T)$.
\end{theorem}
\begin{lemma}\label{lemma1}
Let $T \in B(X)$, then 

(i) $\sigma_{uf}(T)=\sigma_{ub}(T) \Leftrightarrow \sigma_{usbf}(T)=\sigma_{usbb}(T)$,

(ii)  $\sigma_{lf}(T)=\sigma_{lb}(T) \Leftrightarrow \sigma_{lsbf}(T)=\sigma_{lsbb}(T)$.
\end{lemma}
\begin{proof}
(i) Let $\sigma_{uf}(T)=\sigma_{ub}(T)$. It suffices to show that $\sigma_{usbb}(T)=\sigma_{usbf}(T)$. Let $\lambda_{0} \notin \sigma_{usbf}(T)$. Then $\lambda_{0} I-T$ is upper semi B-Fredholm. Therefore, by \cite[Theorem 1.117]{1} there exists an open disc $\mathbb{D}$ centered at $\lambda_{0}$ such that  $\lambda I-T$ is upper semi-Fredholm for all $\lambda \in \mathbb{D} \setminus \{\lambda_{0}\}$. Since  $\sigma_{uf}(T)=\sigma_{ub}(T)$, $\lambda I-T$ is upper semi-Browder for all $\lambda \in \mathbb{D} \setminus \{\lambda_{0}\}$. Therefore, $p(\lambda I-T) <  \infty$ for all $\lambda \in \mathbb{D} \setminus \{\lambda_{0}\}$. Thus, $T$ has SVEP at every $ \lambda \in \mathbb{D} \setminus \{\lambda_{0}\}$ which gives $T$ has SVEP at $\lambda_{0}$. Thus, by \cite[Theorem 2.5]{33} it follows that $\lambda \notin \sigma_{usbb}(T)$. Conversely, let  $\sigma_{usbb}(T)=\sigma_{usbf}(T)$. It suffices to show that $\sigma_{ub}(T) \subset \sigma_{uf}(T)$. Let $\lambda \notin \sigma_{uf}(T)$. Then $\lambda \notin \sigma_{usbf}(T)=\sigma_{usbb}(T)$. Therefore, $p(\lambda I-T) < \infty$ which implies that $\lambda \notin \sigma_{ub}(T)$.
(ii) Using a similar argument as above we can get the desired result.
\end{proof}
\begin{theorem}\label{theorem5}
Let $T \in B(X)$, then the following statements are equivalent:

(i) $\sigma_{usbf}(T)=\sigma_{usbb}(T)$,

(ii) $T$ has SVEP at every $\lambda \notin\sigma_{usbf}(T)$,

(iii) $T$ has SVEP at every $\lambda \notin\sigma_{gDM\phi_{+}}(T)$,

(iv) $\sigma_{gDM\mathcal{J}}(T)=\sigma_{gDM\phi_{+}}(T)$.
\end{theorem}
\begin{proof}
(i) $\Leftrightarrow$ (ii) Suppose that $\sigma_{usbf}(T)=\sigma_{usbb}(T)$. Let $\lambda \notin \sigma_{usbf}(T)$, then $\lambda \notin \sigma_{usbb}(T)$ which gives $p(\lambda I-T)< \infty$. Therefore, $T$ has SVEP at $\lambda$. Now suppose that $T$ has SVEP at every $\lambda \notin \sigma_{usbf}(T)$.  It suffices to prove that $\sigma_{usbb}(T) \subset \sigma_{usbf}(T)$.  Let $\lambda \notin \sigma_{usbf}(T)$, then $\lambda I-T$ is upper semi B-Fredholm operator. Since $T$ has SVEP at $\lambda$ then  by \cite[Theorem 2.5]{33}  it follows that $\lambda \notin \sigma_{usbb}(T)$.\\
(iii) $\Leftrightarrow$ (iv) Suppose that $T$ has SVEP at every $\lambda \notin \sigma_{gDM\phi_{+}}(T)$ which implies that $\lambda I-T$ is generalized Drazin-meromorphic upper semi-Fredholm. It suffices to show that $\sigma_{gDM\mathcal{J}}(T) \subset \sigma_{gDM\phi_{+}}(T)$. Let $\lambda \notin \sigma_{gDM\phi_{+}}(T)$, then by Theorem \ref{pretheorem1} there exists $(M,N) \in Red(\lambda I-T)$ such that $(\lambda I-T)_M$ is semi-regular  and $(\lambda I-T)_N$ is meromorphic. Since $T$ has SVEP at $\lambda$, $(\lambda I-T)_M$ has SVEP at $0$. Therefore, by \cite[Theorem 2.91]{1} $(\lambda I-T)_M$ is bounded below. Thus, $\lambda \notin \sigma_{gDM\mathcal{J}}(T)$.  Conversely, suppose that $\sigma_{gDM\mathcal{J}}(T)=\sigma_{gDM\phi_{+}}(T)$.  Let $\lambda \notin \sigma_{gDR\phi_{+}}(T)$, then $ \lambda I-T$ is generalized Drazin-meromorphic bounded below. Therefore, by  \cite[Theorem 2.5]{30} it follows that $T$ has SVEP at $\lambda$.\\
(i) $\Leftrightarrow$ (iv) Suppose that $\sigma_{usbf}(T)=\sigma_{usbb}(T)$.  It suffices to prove that $\sigma_{gDM\mathcal{J}}(T) \subset \sigma_{gDM\phi_{+}}(T)$. Let $\lambda \notin \sigma_{gDM\phi_{+}}(T)$, then $\lambda I-T$ is generalized Drazin-meromorphic upper semi-Fredholm. By Theorem \ref{pretheorem1} it follows that $\lambda I-T$ admits a $GKMD$ and $\lambda \notin \mbox{acc} \sigma_{usbf}(T)$. This gives $\lambda \notin \mbox{acc}\sigma_{usbb}(T)$. Therefore, by  \cite[Theorem 2.5]{30} $\lambda I-T$ is generalized Drazin-meromorphic bounded below which gives $\lambda \notin \sigma_{gDM\mathcal{J}}(T)$.  Conversely, suppose that $\sigma_{gDM\mathcal{J}}(T)=\sigma_{gDM\phi_{+}}(T)$. Then by (iv) $\Rightarrow$ (iii) $T$ has SVEP at every $\lambda \notin \sigma_{gDM\phi_{+}}(T)$. Since $\sigma_{gDM\phi_{+}}(T) \subset \sigma_{uf}(T)$, $T$ has SVEP at every $\lambda \notin \sigma_{uf}(T)$. Therefore, by \cite[Theorem 2.8]{27} we have $\sigma_{uf}=\sigma_{ub}(T)$. Thus, by Lemma \ref{lemma1} $\sigma_{usbf}=\sigma_{usbb}(T)$. 
\end{proof}
\begin{theorem}\label{theorem6}
Let $T \in B(X)$, then the following statements are equivalent:

(i) $\sigma_{lsbf}(T)=\sigma_{lsbb}(T)$,

(ii) $T^*$ has SVEP at every $\lambda \notin\sigma_{lsbf}(T)$,

(iii) $T^*$ has SVEP at every $\lambda \notin\sigma_{gDM\phi_{-}}(T)$,

(iv) $\sigma_{gDM\mathcal{Q}}(T)=\sigma_{gDM\phi_{-}}(T)$.
\end{theorem}
\begin{proof}
(i) $\Leftrightarrow$ (ii) Suppose that $\sigma_{lsbf}(T)=\sigma_{lsbb}(T)$. Let $\lambda \notin \sigma_{lsbf}(T)$, then $\lambda \notin \sigma_{lsbb}(T)$ which gives $q(\lambda I-T)< \infty$. Therefore, $T^*$ has SVEP at $\lambda$. Now suppose that $T^*$ has SVEP at every $\lambda \notin \sigma_{lsbf}(T)$.  It suffices to prove that $\sigma_{lsbb}(T) \subset \sigma_{lsbf}(T)$.  Let $\lambda \notin \sigma_{lsbf}(T)$, then $\lambda I-T$ is lower semi B-Fredholm operator. Since $T^*$ has SVEP at $\lambda$ then  by  \cite[Theorem 2.5]{33} we have $\lambda \notin \sigma_{lsbb}(T)$.\\
(iii) $\Leftrightarrow$ (iv) Suppose that $T^*$ has SVEP at every $\lambda \notin \sigma_{gDM\phi_{-}}(T)$ which implies that $\lambda I-T$ is generalized Drazin-meromorphic lower semi-Fredholm. It suffices to show that $\sigma_{gDM\mathcal{Q}}(T) \subset \sigma_{gDM\phi_{-}}(T)$. By Theorem \ref{pretheorem1} it follows that $\lambda I-T$ admits a $GKMD$ and $\lambda \notin \mbox{acc} \sigma_{lsbf}(T)$. Since $\sigma_{gDM\phi_{-}}(T) \subset \sigma_{lf}(T)$, $T^*$ has SVEP at every $\lambda \notin \sigma_{lf}(T)$. Therefore, by \cite[Theorem 2.9]{27} we have $\sigma_{lf}=\sigma_{lb}(T)$. Thus, by Lemma \ref{lemma1} b  $\sigma_{lsbf}=\sigma_{lsbb}(T)$ which implies that $\lambda \notin \mbox{acc} \sigma_{lsbb}(T)$.  Hence, $\lambda \notin \sigma_{gDM\mathcal{Q}}(T).$ Conversely, suppose that $\sigma_{gDM\mathcal{Q}}(T)=\sigma_{gDM\phi_{-}}(T)$.  Let $\lambda \notin \sigma_{gDM\phi_{-}}(T)$, then $ \lambda I-T$ is generalized Drazin-meromorphic surjective. Therefore by  \cite[Theorem 2.6]{30} it follows that $T^{*}$ has SVEP at $\lambda$.\\
(i) $\Leftrightarrow$ (iv) Suppose that $\sigma_{lsbf}(T)=\sigma_{lsbb}(T)$.  It suffices to prove that $\sigma_{gDM\mathcal{Q}}(T) \subset \sigma_{gDM\phi_{-}}(T)$. Let $\lambda \notin \sigma_{gDM\phi_{-}}(T)$, then $\lambda I-T$ is generalized Drazin-meromorphic lower semi-Fredholm. By Theorem \ref{pretheorem1} it follows that $\lambda I-T$ admits a $GKMD$ and $\lambda \notin \mbox{acc} \sigma_{lsbf}(T)$. This gives $\lambda \notin \mbox{acc} \sigma_{lsbb}(T)$. Therefore, by  \cite[Theorem 2.6]{30} $\lambda I-T$ is generalized Drazin-meromorphic surjective which gives $\lambda \notin \sigma_{gDM\mathcal{Q}}(T)$.  Conversely, suppose that $\sigma_{gDM\mathcal{Q}}(T)=\sigma_{gDM\phi_{-}}(T)$. Then by (iv) $\Rightarrow$ (iii) $T^{*}$ has SVEP at every $\lambda \notin \sigma_{gDM\phi_{-}}(T)$. Since $\sigma_{gDM\phi_{-}}(T) \subset \sigma_{lf}(T)$ , $T^{*}$ has SVEP at every $\lambda \notin \sigma_{lf}(T)$. This gives $\sigma_{lsbf}(T)=\sigma_{lsbb}(T)$.
\end{proof}
Using \cite[Corollary 2.10]{27} and Theorems \ref{theorem5}, \ref{theorem6}  we have the following result:
\begin{corollary}
Let $T \in B(X)$, then the following statements are equivalent:

(i) $\sigma_{f}(T)=\sigma_{b}(T)$,

(ii) $T$ and $T^{*}$ have SVEP at every $\lambda \notin \sigma_{f}(T)$,

(iii) $\sigma_{bf}(T)=\sigma_{bb}(T)$,

(iv) $T$ and $T^{*}$ have SVEP at every $\lambda \notin \sigma_{bf}(T)$,

(v) $\sigma_{gD}(T)=\sigma_{pbf}(T)$,

(vi) $T$ and $T^{*}$ have SVEP at every $\lambda \notin \sigma_{pbf}(T)$,

(viii) $\sigma_{gDR}(T)=\sigma_{gDR\phi}(T)$,

(viii) $T$ and $T^{*}$ have SVEP at every $\lambda \notin \sigma_{gDR\phi}(T)$,

(ix) $\sigma_{gDM}(T)=\sigma_{gDM\phi}(T)$,

(x) $T$ and $T^{*}$ have SVEP at every $\lambda \notin \sigma_{gDM\phi}(T)$.
\end{corollary}
 \section{Cline's Formula for the generalized Drazin-meromorphic invertibility}
 Let $R$ be a ring with identity. Drazin\cite{13} introduced the concept of Drazin inverses in a ring. 
 An element $a \in R$ is said to be \emph{Drazin invertible} if there exist an element $b \in R$ and $r \in \mathbb{N}$ such that 
 $$ a b=b a,\thinspace bab=b,\thinspace a^{r+1} b=a^r.$$
 If such $b$ exists then it is unique and is called \emph{Drazin inverse} of $a$ and denoted by $a^D$. For $a ,b \in R$, Cline \cite{12}  proved  that if $a b$ is Drazin invertible, then $b a$ is Drazin invertible and $(b a)^D=b((a b)^D)^2 a$. Recently, Gupta and Kumar \cite{3} generalized Cline's formula for Drazin inverses in a ring with identity to the case when $a^kb^ka^k=a^{k+1}$ for some $k \in \mathbb{N}$ and obtained the following result:
 \begin{theorem} \emph{(\cite[Theorem 2.20]{3})} 
Let $R$ be a ring with identity and suppose that $a^k b^k a^k=a^{k+1} $ for some $k \in \mathbb{N}$. Then
$a$ is Drazin invertible if and only if $b^k a^k$ is Drazin invertible. Moreover, $(b^k a^k)^D = b^k (a^D)^2 a^k$ and $a^D=a^k (b^k a^k)^D)^{k+1}$.
\end{theorem}  

Recently, Karmouni and Tajmouati \cite{7}  investigated for bounded linear operators $A,B,C$ satisfying the operator equation $ABA=ACA$ and  obtained that $AC$ is generalized Drazin-Riesz invertible if and only if $BA$ is generalized Drazin-Riesz invertible. Also, they generalized Cline's formula to the case of generalized Drazin-Riesz invertibility. In this section, we establish Cline's formula for the generalized  Drazin- Riesz invertibility for bounded linear operators $A$ and $B$ under the condition $A^k B^k A^k=A^{k+1}$.  By \cite[Theorem 2.1, Theorem 2.2, Proposition 2.4 and Lemma 2.1]{3} and a result \cite[Corollary 3.99]{1} we can deduce the following result:
 \begin{proposition}\label{proposition3}
 Let $A,B \in B(X)$ satisfies  $A^k B^k A^k=A^{k+1}$ for some $k \in \mathbb{N}$, then $A$ is meromorphic if and only if $B^kA^k$ is meromorphic.
 \end{proposition}
 \begin{theorem}
 Suppose that $A,B \in B(X)$ and  $A^k B^k A^k=A^{k+1}$ for some $k \in \mathbb{N}$. Then $A$ is generalized Drazin-meromorphic invertible if and only if $B^k A^k$ is generalized Drazin-meromorphic invertible.
 \end{theorem}
 \begin{proof}
 Suppose that $A$ is generalized Drazin-meromorphic invertible, then there exists $T \in B(X)$ such that 
 $$TA=AT, \quad TAT=T \quad \mbox{and} \quad ATA-A \thinspace \thinspace  \mbox{is meromorphic}.$$
  Let $S=B^k T^2 A^k.$ Then  $$(B^k A^k)S=(B^k A^k)(B^k T^2 A^k)=B^k(A^k B^k A^k)T^2=B^k A^{k+1} T^2=B^k A^k T$$ and $$S(B^k A^k)=(B^k T^2 A^k) (B^k A^k)=B^k T^2 A^{k+1}=B^k A^k T.$$ Therefore, $S(B^k A^k)=(B^k A^k)S$. Consider 
  \begin{align*}
  S(B^k A^k) S&=B^k T^2 A^k(B^k A^k)B^k T^2 A^k=(B^k T^2 A^k)(B^k  A^k T)=B^k T^2 A^{k+1}T=B^k T^2 A^k=S.
  \end{align*}
   Let $Q=I-AT$, then $Q$ is a bounded projection commuting with $A$.  Therefore, $Q^n=Q$ for all $n \in \mathbb{N}$. We observe that 
  $$(QA)^k B^k (QA)^k=Q^k A^k B^k Q^k A^k=Q^k A^{k+1} Q^k=Q^{k+1} A^{k+1}=(QA)^{k+1}$$ and
\begin{align*}
B^k A^k-(B^k A^k)^2 S&=B^k A^k-(B^k A^k)^2 B^k T^2 A^K=B^k A^k-B^k (A^k B^k A^k) B^k T^2 A^k\\
&=B^k A^k-B^k A^{k+2} T^2=B^k (I-A^2 T^2) A^k=B^k (I-A T) A^k\\
&=B^k Q A^k=B^k Q^k A^k=B^k(QA)^k.
\end{align*} Since $QA$ is meromorphic and$(QA)^k B^k (QA)^k=(QA)^{k+1}$, by Proposition \ref{proposition3} $B^k A^k-(B^k A^k)^2 S$ is meromorphic.

Conversely, suppose that $B^k A^k$ is generalized Drazin-meromorphic invertible. Then there exists $T' \in B(X)$ such that
$$T'B^k A^k=B^k A^kT', \quad T'B^k A^kT'=T' \quad \mbox{and} \quad B^k A^kT'B^k A^k-B^k A^k \thinspace \thinspace  \mbox{is meromorphic}.$$ Let $S'=A^k {T'}^{k+1}.$ Then $$S'A=A^k {T'}^{k+1} A=A^k {T'}^{k+2} B^k A^k A=A^k {T'}^{k+2} B^k A^{k+1}=A^k {T'}^{k+2} (B^k A^k)^2=A^k {T'}^k $$ and $$AS'=A^{k+1} {T'}^{k+1}=A^k {T'}^k.$$
Consider
\begin{align*}
AS'&=(A^k {T'}^{k+1} A) A^k {T'}^{k+1}=(A^k {T'}^k)A^k {T'}^{k+1}=A^k v^{k+1}B^k A^{2 k} {T'}^{k+1}=A^k {T'}^{k+1}(B^k A^k)^{k+1}\\
&= S^{k+1}=A^k {T'}^{k+1}=S'.
\end{align*}
We claim that for all $n \in \mathbb{N}$ we have  $$(A-A^2 S')^n= (A^n-A^{n+1} S').$$
We prove it by induction. Evidently, the result is true for $n=1$. Assume it to be true for $n=p$. Consider 
\begin{align*}
(A-A^2 S')^{p+1}&=(A-A^2 S')(A-A^2 S')^p\\
&=(A-A^2 S')(A^p-A^{p+1} S')\\
&=A^{P+1}-A^{P+2}S'-A^{P+2}S'+A^{P+3} {S'}^2\\
&=A^{p+1}-A^{p+2} S'.
\end{align*} 
Also, 
\begin{align*}
B^k(A-A^2 S')^k&=B^k(A^k- A^{k+1} S')=B^k A^k-B^k A^{k-1}A^2 S'=B^k A^k-B^k A^{k-1}A^k {T'}^{k-1}\\&=B^k A^k-B^k A^{2k-1}{T'}^{k-1}=B^k A^k-(B^k A^{k})^k {T'}^{k-1}=B^k A^k-(B^k A^{k})^2 S'.
\end{align*}
Now consider 
\begin{align*}
(A-A^2 S')^k B^k (A-A^2 S')^k&=(A^k-A^{k+1} S') B^k (A^k-A^{k+1} S')\\
&=A^k B^k A^k-A^{k+1} S' B^k A^k-A^k B^k A^k  B^k A^k S'+A^{k+1} (B^k A^k)^2 {S'}^2\\
&=A^{k+1}-A^{k+2}S'=(A-A^2 S')^{k+1}.
\end{align*}  Since $B^k(A-A^2 S')^k=B^k A^k-(B^k A^k)^2 T'$ is meromorphic, by Proposition \ref{proposition3}  it follows that $A-A^2 S'$ is meromorphic.
 \end{proof}
 \section*{Acknowledgement}
 The second author is supported by Department of Science and Technology, New Delhi, India (Grant No. DST/INSPIRE Fellowship/[IF170390]).

 \end{document}